\documentclass[12pt,reqno]{amsart}
\usepackage{amssymb,mathrsfs,graphicx}
\usepackage{mathtools}
\usepackage{pifont,color}
\usepackage[margin=3cm]{geometry}
\mathtoolsset{showonlyrefs=true}

\title[Critical threshold for a traffic flow model]{
  A sharp critical threshold for a traffic flow model with look-ahead dynamics}

\author[Yongki Lee]{Yongki Lee$^\dag$}
\address{$^\dag$Department of Mathematical Sciences, Georgia Southern
  University,  Statesboro, Georgia 30458}
\email{yongkilee@georgiasouthern.edu}

\author[Changhui Tan]{Changhui Tan$^\ddag$}
\address{$^\ddag$Department of Mathematics, University of South
  Carolina, Columbia, South Carolina 29208}
\email{tan@math.sc.edu}

\subjclass[2010]{35L65, 35L67}

\keywords{nonlocal conservation law, traffic flow,
  critical threshold, global regularity, shock formation}

\thanks{\emph{Acknowledgment.} The research of CT is supported by the NSF grant
DMS 1853001. }

\newtheorem{theorem}{Theorem}[section]
\newtheorem{lemma}{Lemma}[section]

\newtheorem{proposition}{Proposition}[section]
\newtheorem{remark}{Remark}[section]

\def\R{\mathbb{R}}
\def\pa{\partial}
\def\ubar{\bar{u}}

\newcommand{\Rom}[1]
    {\text{\MakeUppercase{\romannumeral #1}}}

\begin{document}
%%%%%%%%%%%%%%%%
\allowdisplaybreaks

\begin{abstract}
  We study a nonlocal traffic flow model with an Arrhenius type
  look-ahead interaction. We show a \emph{sharp} critical threshold
  condition on the initial data which distinguishes the global
  smooth solutions and finite time wave break-down. 
\end{abstract}

\maketitle %\centerline{\date}
%\tableofcontents

\section{Introduction}
We consider the following one-dimensional traffic flow model with a
nonlocal flux
\begin{equation}\label{traffic_main}
  \begin{cases}
    \partial_t u + \partial_x \big( u(1-u)e^{-\ubar} \big) =0, & t>0, x \in \R, \\
    u(0,x)=u_0 (x), &  x\in \R.
  \end{cases}
\end{equation}
Here, $u(t,x)$ represents the vechicle density normalized in the
interval $[0,1]$. The velocity of the flow $v=(1-u)e^{-\ubar}$ becomes
zero when the maximum density is reached. It is also weighted by a
nonlocal Arrhenius type slow down factor $e^{-\ubar}$, where 
\begin{equation}\label{ubar}
\bar{u}(t,x)=(K*u)(t,x)=\int_{\mathbb{R}} K(x-y)u(t,y) \, dy,
\end{equation}
with appropriately choices of the kernel $K$ to be discussed later.

We are interested in the local and global wellposedness of this
nonlocal macroscopic traffic flow model \eqref{traffic_main}-\eqref{ubar}.
The goal is to understand whether smooth solutions presist in all
time, or there is a finite time singularity formation. Such blowup is
known as the \emph{wave break-down phenomenon}, which discribes the
generation of the traffic jam.

\subsection{Nonlocal conservation laws}
The traffic flow model \eqref{traffic_main} falls into a class of
models in nonlocal scaler conservation laws, which has the form
 \begin{equation}\label{class_main}
 \partial_t u + \partial_x F(u, \ubar)=0,
\end{equation}
where the flux $F$ depends on both the local density $u$, and the
nonlocal quantity $\ubar$ defined in \eqref{ubar}.
This class of models has a variety of applications, not only in
traffic flows
\cite{kurganov2009non,lee2015thresholds,sopasakis2006stochastic}, but
also in dispersive water waves
\cite{constantin1998wave,holm2005class, liu2006wave, whitham2011linear},
the collective motion of biological cells
\cite{burger2008asymptotic, dolak2005keller},
high-frequency waves in relaxing medium
\cite{hunter1990numerical, vakhnenko2002calculation},
the kinematic sedimentation
\cite{betancourt2011nonlocal,kynch1952theory,zumbrun1999nonlocal},
and many more.
The understanding of the wave break-down phenomenon is important and
challenging for these models.

Here are several intriguing models that lie in this class \eqref{class_main}.
\begin{itemize}
  \item The Whitham equation in nonlinear water waves \cite{whitham2011linear}
\[
  \pa_tu + \pa_x\left(\alpha u^2+\ubar\right) =0,
\]
where the kernel has its Fourier transform
$\hat{K}(\xi)=\left(\frac{\tanh\xi}{\xi}\right)^{1/2}$.
Wave break-down has been shown in \cite{hur2017wave}, for initial conditions
which are near break-down.
\item A one-dimensional hyperbolic Keller-Segel model \cite{dolak2005keller}
  \[
    \pa_tu+\pa_x\left(u(1-u)\pa_xS\right)=0,\quad -\pa_{xx}^2S+S=u.
  \]
 It is shown in \cite{lee2015threshold} that wave break-down happens
 for a set of supercritical initial conditions.
\item The one-dimensional aggregation equation
\[
   \pa_tu+\pa_x\left(u\ubar\right)=0,
\]
where the kernel $K=-\pa_x\phi$ for some interaction potential $\phi$.
If $\phi$ is attractive, then the solution is globally regular if and
only if the Osgood condition holds
\cite{bertozzi2009blow,bertozzi2011lp,carrillo2011global}.
There will be finite time density consentration if the condition is
violated. For general attractive-repulsive interaction potential,
there will be no density concentration if the repulsion is strong
enough. However, there might be wave break-down in finite time, see
for instance \cite{tan2017singularity}.

% \item A nonlocal dispersive equation modeling particle suspensions
%   \cite{rubinstein1990evolution, rubinstein1989sedimentation}
%   \[
%     \pa_tu + \pa_x\left(u+\ubar u\right)=0.
%   \]
%   with a family of bounded and compactly supported kernels.
%   $K_L(x)=\frac{1}{L}K(\frac{1}{L})$ and
%  \[
%     K(x)=\begin{cases}
%     2/(3(r^2 /4 -1)) & $if$ \ |r| < 2, \\
%     0, &   $otherwise$.\hbox{}
%     \end{cases}
%   \]
% Global wellposedness has been proved in \cite{zumbrun1999nonlocal},
% along with the long time behaviors.  
\end{itemize}

The wave break-down phenomenon for general nonlocal conservation laws
\eqref{class_main} has been recently studied in \cite{lee2018wave}. A
sufficient condition on initial data is derived which guarantees a
finite time blowup.

\subsection{Nonlocal traffic models}
We focus on the
nonlocal traffic models \eqref{traffic_main}-\eqref{ubar}.
It is another example of the nonlocal conservation law \eqref{class_main}.

When there is no interaction, namely $K\equiv0$, the dynamics is the
classical Lighthill-Whitham-Richards (LWR) model
\begin{equation}\label{LWR}
  \pa_tu+\pa_x(u(1-u))=0.
\end{equation}
For this local model, it is well-known that there is a finite time
wave break-down for any smooth initial data.

For uniform interaction $K\equiv1$, the nonlocal term
\[\ubar(t,x)=\int_\R u(t,y)dy=\int_\R u_0(y)dy=:m\]
is a constant, due to the conservation of mass. Then, the dynamics
again becomes LWR model, with velocity $v=(1-u)e^{-m}$.

Another class of choices of $K$ is called the \emph{look-ahead}
kernel, where
\[\text{supp}(K)\subseteq(-\infty, 0].\]
Under the assumption, the nonlocal term
\[\ubar(t, x)=\int_x^\infty K(x-y)u(t,y)dy\]
only depends on the density ahead.
Sopasakis and Katsoulakis (SK) in \cite{sopasakis2006stochastic} introduce
a celebrated traffic model with Arrhenius type look-ahead interactions,
where
\begin{equation}\label{SK}
  K(x)=\begin{cases}
    1&-1<x<0,\\0&\text{otherwise}.
  \end{cases}
\end{equation}
A family of kernel with look-ahead distance $L$ can be generated by
the scaling
\begin{equation}\label{scaling}
  K_L(x)=K\left(\frac{x}{L}\right).
\end{equation}
Note that when taking $L\to0$, the system reduces to the local LWR
model \eqref{LWR}.

The wave break-down phenomenon for the SK model is observed in
\cite{kurganov2009non}, through an extensive
numerical study. A different class of linear look-ahead kernel is also
introduced, with
\begin{equation}\label{linear}
  K(x)=\begin{cases}
    2\big(1-(-x)\big)&-1<x<0,\\0&\text{otherwise}.
  \end{cases}
\end{equation}
Numerical examples suggest that wave break-down happens in finite
time, for a class of initial data. However, unlike the LWR model, it
is generally unclear for the nonlocal models whether wave break-down
happens for all smooth initial data. 

\subsection{Critical threshold and wave break-down}
In many examples above, whether there is a finite time wave break-down
depends on the choice of initial conditions: subcritical initial data
lead to global smooth solution, while supercritical initial data
lead to a finite time wave break-down. 
This is known as the \emph{critical threshold phenomenon}, which has
been studied in the context of Eulerian dynamics, including the
Euler-Poisson equations
\cite{engelberg2001critical,lee2013thresholds,tadmor2003critical},
the Euler-Alignment equations
\cite{carrillo2016critical,tadmor2014critical,tan2019euler}, and more
systems of conservation laws.

A critical threshold is called \emph{sharp} if all initial data lie in
either the subcritical region, or the supercritical region.

For the traffic model \eqref{traffic_main} with nonlocal look-ahead
interactions \eqref{SK} or \eqref{linear}, a supercritical region has
been obtained in \cite{lee2015thresholds}. which leads to a finite
time wave break-down. However, the result is not sharp. In particular,
a challenging open question is, whether there exists subcritical initial
data, such that the solution is globally regular.  

\subsection{Main result}
We study the traffic flow model \eqref{traffic_main} with the
following look-ahead interaction
\begin{equation}\label{kernel}
  K(x)=\begin{cases}
    1&-\infty<x<0,\\0&\text{otherwise}.
  \end{cases}
\end{equation}
The kernel can be viewed as a  limit of the SK model \eqref{SK} under
scaling \eqref{scaling}, with look-ahead distance $L\to\infty$.

The corresponding nonlocal term is given by
\begin{equation}\label{ourubar}
 \ubar(t,x)=\int_x^\infty u(t,y)dy.
\end{equation}

 The main result is stated as follow:

\begin{theorem}[Sharp critical threshold]\label{thm:main}
Consider the traffic flow model \eqref{traffic_main} with
a nonlocal look-ahead kernel \eqref{ourubar}.
Suppose the initial data is smooth, with $u_0\in L^1\cap H^s(\R)$ for
$s>3/2$, and $0\leq u_0\leq1$. 
Let $\sigma$ be a function defined in \eqref{eq:sigma}.
Then,
\begin{itemize}
  \item If the initial data is \textbf{subcritical}, satisfying
    \begin{equation}\label{sub}
      u_0'(x)\leq\sigma(u_0(x)),\quad \forall~x\in\R,
    \end{equation}
    then the  solution exists globally in time. Namely, for any $T>0$, there
    exists a
    unique solution $u\in C([0,T]; L^1\cap H^s(\R))$.
  \item If the initial data is \textbf{supercritical}, satisfying
    \begin{equation}\label{super}
      \exists~x_0\in\R\quad s.t.\quad u_0'(x_0)>\sigma(u_0(x_0)),
    \end{equation}
    then the solution must blow up in finite time. More precisely,
    there exists a finite time $T_*>0$, such that
    \[\limsup_{t\to T_*}\|\pa_xu(t,\cdot)\|_{L^\infty}=+\infty.\]
\end{itemize}
\end{theorem}

\begin{remark}
To the best of our knowledge, this is the first result for the
nonlocal traffic models where wave break-down does not happen for a
class of subcritical initial data. 

An example of subcritical initial data is given in Section
\ref{sec:expsub}. Global regularity is verified through numerical
simulation. A striking discovery is, with this initial condition, 
finite time wave break-downs are observed both the LWR model
and the SK model. This indicates a unique feature of the kernel
\eqref{ourubar}.
\end{remark}

\begin{remark}
The critical threshold result in Theorem \ref{thm:main} is sharp. For
nonlocal conservation laws, sharp results are usually hard to obtain,
due to the presence of nonlocality. We utilize a special structure
of the kernel \eqref{ourubar} to obtain a sharp threshold,
$\pa_x\ubar=-u$. So, this kernel is in some sense more ``local''.
Possible extensions for more general kernels will be discussed in
Section \ref{sec:close}.
\end{remark}

The rest of the paper is organized as follows. In Section
\ref{sec:local}, we establish the local wellposedness theory for our
nonlocal traffic model \eqref{traffic_main} with \eqref{ourubar}, as well as a criterion to
preserve smooth solutions.
In Section \ref{sec:threshold}, we show the sharp critical threshold,
and prove Theorem \ref{thm:main}.
Some numerical examples are provided in Section \ref{sec:exp}, which
illustrate the behaviors of the solution under subcritical and
supercritical initial data. Finally, we make some remarks in Section
\ref{sec:close}, which would lead to future investigations.

\section{Local wellposedness and regularity criterion}\label{sec:local}
In this section, we establish the local wellposedness theory for our
main system \eqref{traffic_main}.

\begin{theorem}[Local wellposedness]\label{thm:local}
Let $s>3/2$. Consider equation \eqref{traffic_main} with
\eqref{ourubar}. Suppose the initial data $u_0\in L^1\cap H^s(\R)$, and
$0\leq u_0\leq1$. Then, 
there exists a time $T_*=T_*(u_0)>0$, such that the solution $u(t,x)$
exists in $L^\infty([0,T]; L^1\cap H^s(\R))$.

Moreover, for any time $T>0$, the solution exists in $L^\infty([0,T];
L^1\cap H^s(\R))$ if and only if
\begin{equation}\label{BKM}
\int_0^T\|\pa_xu(\cdot,t)\|_{L^\infty}dt<+\infty.
\end{equation}
\end{theorem}

\subsection{Conservation of mass}
Assume $u$ vanishes at infinity.
Integrating \eqref{traffic_main} in $x$, we obtain 
\[\frac{d}{dt}\int_\R u(t,x)dx=-\int_\R\pa_x(u(1-u)e^{-\ubar})dx=0.\]
Therefore, the total mass
\[m:=\int_\R u(t,x)dx\]
is conserved in time. From \eqref{ourubar}, we get the following a
priori bound on $\ubar$
\begin{equation}\label{eq:ubarbound}
0 \leq \ubar(t,x)\leq m.
\end{equation}

\subsection{Maximum principle}
We next show that there is a maximum density for our traffic model.
Rewrite  \eqref{traffic_main} as
\begin{equation}\label{eq:main}
\pa_tu+(1-2u)e^{-\ubar}\pa_xu+u^2(1-u)e^{-\ubar}=0.
\end{equation}
Let $X(t)=X(t;x)$ be the characterstic path originated at $x$, defined as
\[\frac{d}{dt}X(t;x)=(1-2u(t,X(t;x)))e^{-\ubar(t,X(t;x))},\quad X(0;x)=x.\]
Then, along each  characterstic path
\begin{equation}\label{eq:uchar}
  \frac{d}{dt}u(t, X(t))=-u^2(1-u)e^{-\ubar},
\end{equation}
where the right hand side is evaluated at $(t, X(t))$.

The following maximum principle holds.
\begin{proposition}[Maximum principle]
Let $u$ be a classical solution of \eqref{eq:main}, with
initial condition $0\leq u_0\leq 1$. Then, $0\leq u(x,t)\leq1$ for any
$x\in\R$ and $t\geq0$.
\end{proposition}
\begin{proof}
  Suppose there exist a positive time $t>0$ and a characteristic path such that
  $u(t,X(t))>1$. Then, there must be a time $t_0$ when the first
  breakdown happens, namely
 \[u(t_0, X(t_0))=1,\quad u(t_0+, X(t_0+))>1.\]
 However, solving the initial value problem \eqref{eq:uchar} with 
 $u(t_0, X(t_0))=1$, we obtain
 \[u(t,X(t))=1,\quad\forall~t\geq t_0.\]
 This leads to a contradiction. Hence, $u(x,t)\leq 1$ for any $x$ and
 $t\geq0$.
The preservation of positivity $u(x,t)\geq0$ can be proved using the
same argument.
\end{proof}

\subsection{A priori bounds on the nonlocal term}
We now bound the nonlocal term $e^{-\ubar}$.
First, from \eqref{eq:ubarbound}, we have
\begin{equation}\label{eq:nonlocalLinf}
  e^{-m}\leq e^{-\ubar} \leq 1.
\end{equation}
This shows the nonlocal weight is bounded from above and below, away
from zero.

Next, we compute
\begin{equation}\label{eq:nonlocalW1inf}
\|\pa_x(e^{-\ubar})\|_{L^\infty}=\|u\cdot
e^{-\ubar}\|_{L^\infty}\leq 1.
\end{equation}

For higher derivatives of $e^{-\ubar}$, we have the following estimate.
\begin{proposition}\label{prop:nonlocal} For $s\geq1$,
\[\|e^{-\ubar}\|_{\dot{H}^s}\lesssim\|u\|_{\dot{H}^{s-1}}.\]
\end{proposition}
\begin{proof}
  We apply the composition estimate, stated and proved in Theorem
  \ref{thm:composition}, with 
  $f(x)=e^x$ and $g(x)=-\ubar(t,x)$.

  From \eqref{eq:ubarbound}, we know $g$ is bounded, and $g(x)\in[-m,0]$. Therefore,
  $\|f\|_{C^s([-m,0])}=1$ for any $s\in\mathbb{N}$.

  Theorem \ref{thm:composition} implies
  \[\|e^{-\ubar}\|_{\dot{H}^s}\lesssim\|g\|_{\dot{H}^s}
  =\|u\|_{\dot{H}^{s-1}}.\]
 The last equality is due to the fact that $\pa_xg=u$. 
\end{proof}

\subsection{$L^2$ energy estimate}
We perform a standard $L^2$ energy estimate.
\begin{align}
\frac{1}{2}\frac{d}{dt}\|u(\cdot,t)\|_{L^2}^2=&
-\int_\R u~\pa_x\big(u(1-u)e^{-\ubar}\big)dx
=\int_\R \pa_xu\cdot u(1-u)e^{-\ubar}dx
\nonumber\\
=&-\int_\R \frac{1}{2}u^2\cdot\pa_x(e^{-\ubar})dx -
\int_\R u^2\cdot\pa_xu\cdot e^{-\ubar} dx \nonumber\\
\leq&~
\frac{1}{2}\|u\|_{L^2}^2\|\pa_xe^{-\ubar}\|_{L^\infty}+\|\pa_xu\|_{L^\infty}
\|u\|_{L^2}^2\|e^{-\ubar}\|_{L^\infty}\nonumber\\
\lesssim&~(1+\|\pa_xu\|_{L^\infty})\|u\|_{L^2}^2,\label{eq:L2est}
\end{align}
where we apply \eqref{eq:nonlocalLinf} and \eqref{eq:nonlocalW1inf} in
the last inequality.

A simple Gronwall-type estimate then yields
\[\|u(\cdot,t)\|_{L^2}\leq \|u_0\|_{L^2}\exp
\left(C\int_0^t(1+\|\pa_xu(\cdot,\tau)\|_{L^\infty})d\tau\right).\]
Hence, $u(\cdot,t)\in L^2$ for $t\in[0,T]$ as long as \eqref{BKM} holds. 

\subsection{$H^s$ energy estimate}
Let $\Lambda :=(-\Delta)^{1/2}$ be the pseudo-differential
operator. We perform an energy estimate by acting $\Lambda^s$ on
\eqref{eq:main} and integrate against $\Lambda^s u$. This yields the
evolution of the homogeneous $H^s$-norm on $u$:
\begin{align*}
\frac{1}{2}\frac{d}{dt}\|u(\cdot,t)\|_{\dot{H}^s}^2=&
\int_\R \Lambda^su \cdot\Lambda^s\big(
-(1-2u)e^{-\ubar}\pa_xu-u^2(1-u)e^{-\ubar}\big)dx\\
=&\int_\R \Lambda^su \cdot (2u-1)e^{-\ubar}\cdot\Lambda^s
\pa_xu~dx
+\int \Lambda^su\cdot\big(\left[\Lambda^s,
   (2u-1)e^{-\ubar}\right]\pa_xu \big)dx\\
-&\int_\R \Lambda^su \cdot \big(u^2(1-u)e^{-\ubar}\big)dx
=\Rom{1} + \Rom{2} +\Rom{3}.
\end{align*}
Here, the commutator $[\Lambda^s, f]g$ is defined as
\[ [\Lambda^s, f]g = \Lambda^s(fg)-f \Lambda^sg.\]

We shall estimate the three terms one by one. 

For the first term, apply integration by parts and get
\begin{align*}
\Rom{1}=&\int_\R\frac{1}{2}\pa_x\big((\Lambda^su)^2\big)\cdot (2u-1)e^{-\ubar}~dx
=-\frac{1}{2}\int_\R(\Lambda^su)^2\cdot\pa_x((2u-1)e^{-\ubar})dx\\
\leq&~\frac{1}{2}\|u\|_{\dot{H}^s}^2\|\pa_x((2u-1)e^{-\ubar})\|_{L^\infty}
=\frac{1}{2}\|u\|_{\dot{H}^s}^2\|(2\pa_xu+(2u-1)u)e^{-\ubar}\|_{L^\infty}.
\end{align*}
Since both $u$ and $\ubar$ are bounded, we have
\[\|(2\pa_xu+(2u-1)u)e^{-\ubar}\|_{L^\infty}\leq
2\|\pa_xu\|_{L^\infty}+1.\]
Therefore,
\begin{equation}\label{eq:est1}
\Rom{1}\leq(1+\|\pa_xu\|_{L^\infty})\|u\|_{\dot{H}^s}^2.
\end{equation}

For the second term,
\[\Rom{2}\leq\|u\|_{\dot{H}^s}\big\|\left[\Lambda^s,
   (2u-1)e^{-\ubar}\right]\pa_xu\big\|_{L^2}.\]
Let us state the following two estimates. Both lemmas can be proved
using Littlewood-Paley theory.
\begin{lemma}[Fractional Leibniz rule]\label{lem:prod}
 Let $s\geq0$. There exists a constant $C>0$, depending only on $s$,
 such that
 \[\| fg \|_{\dot{H}^s}\leq C
 \left(\|f\|_{L^\infty}\|g\|_{\dot{H}^s}+\|f\|_{\dot{H}^s}\|g\|_{L^\infty}\right).\]
\end{lemma}
A proof of the Fractional Leibniz rule can be found in \cite[Corollary
2.86]{bahouri2011fourier}.

\begin{lemma}[Commutator estimate]\label{lem:com}
 Let $s\geq1$. There exists a constant $C>0$, depending only on $s$,
 such that
 \[\| [\Lambda^s, f]g \|_{L^2}\leq C
 \left(\|\pa_xf\|_{L^\infty}\|g\|_{\dot{H}^{s-1}}+\|f\|_{\dot{H}^s}\|g\|_{L^\infty}\right).\]
\end{lemma}
The commutator estimate is due to Kato and Ponce \cite{kato1988commutator}. See
\cite[Remark 1.5]{li2019kato} for the version for homogeneous operator $\Lambda^s$.

Apply Lemma \ref{lem:com} to the commutator in \Rom{2}. We get
\begin{align*}
\big\|\big[\Lambda^s, &(2u-1)e^{-\ubar}\big]\pa_xu\big\|_{L^2}\lesssim\\
&\|(2u-1)e^{-\ubar}\|_{L^\infty}\|\pa_xu\|_{\dot{H}^{s-1}}+
\| (2u-1)e^{-\ubar}\|_{\dot{H}^s}\|\pa_xu\|_{L^\infty}=\Rom{4}+\Rom{5}.
\end{align*}

Due to maximum principle, $|2u-1|\leq1$. Also,
$\|e^{-\ubar}\|_{L^\infty}\leq1$ by \eqref{eq:nonlocalLinf}. Therefore, \Rom{4} can be
easily estimated by 
\[\Rom{4}\leq\|u\|_{\dot{H}^s}.\]

For \Rom{5}, we apply Lemma \ref{lem:prod} and
Proposition \ref{prop:nonlocal},
\begin{align*}
\Rom{5}\lesssim &\left(\|2u-1\|_{\dot{H^s}}\|e^{-\ubar}\|_{L^\infty}+
\|2u-1\|_{L^\infty}\|e^{-\ubar}\|_{\dot{H^s}}\right)\|\pa_xu\|_{L^\infty}\\
\lesssim&
          \left(\|u\|_{\dot{H}^s}+\|u\|_{\dot{H}^{s-1}}\right)\|\pa_xu\|_{L^\infty}.
\end{align*}

Combine the estimates on \Rom{4} and \Rom{5}, we obtain
\begin{equation}\label{eq:est2}
\Rom{2}\lesssim \|\pa_xu\|_{L^\infty}\|u\|_{\dot{H}^s}\|u\|_{H^s}.
\end{equation}

For the third term, we again apply Lemma \ref{lem:prod} and get
\[\Rom{3}\lesssim\|u\|_{\dot{H}^s}\big(\|u^2(1-u)\|_{\dot{H^s}}\|e^{-\ubar}\|_{L^\infty}+
\|u^2(1-u)\|_{L^\infty}\|e^{-\ubar}\|_{\dot{H}^s}\big).\]
The first part can be further estimated by
\[\|u^2(1-u)\|_{\dot{H^s}}\lesssim2\|u\|_{\dot{H}^s}\|u\|_{L^\infty}\|1-u\|_{L^\infty}
+\|u\|_{L^\infty}^2\|1-u\|_{\dot{H}^s}\lesssim \|u\|_{\dot{H}^s}.\]
Applying Proposition \ref{prop:nonlocal} to the second part, we obtain
\begin{equation}\label{eq:est3}
\Rom{3}\lesssim \|\pa_xu\|_{L^\infty}\|u\|_{\dot{H}^s}\|u\|_{H^s}.
\end{equation}

Gathering the estimates \eqref{eq:est1}, \eqref{eq:est2} and
\eqref{eq:est3}, we derive
\[\frac{d}{dt}\|u(\cdot,t)\|_{\dot{H}^s}^2\lesssim
\|u\|_{\dot{H}^s}^2+ \|\pa_xu\|_{L^\infty}\|u\|_{\dot{H}^s}\|u\|_{H^s}.\]

Together with the $L^2$ estimate \eqref{eq:L2est}, we get the full
$H^s$ estimate
\[\frac{d}{dt}\|u(\cdot,t)\|_{H^s}^2\lesssim
(1+ \|\pa_xu\|_{L^\infty})\|u\|_{H^s}^2.\]
Applying Gronwall inequality, we end up with
\[\|u(\cdot,t)\|_{H^s}\leq \|u_0\|_{H^s}\exp
\left(C\int_0^t(1+\|\pa_xu(\cdot,\tau)\|_{L^\infty})d\tau\right).\]

For $s>3/2$, $H^s$ is embedded in $W^{1,\infty}$. Therefore, if
$u_0\in H^s$, then $\|u_0'\|_{L^\infty}$ is bounded. The solution
exists locally in time. Moreover, $u(\cdot,t)\in H^s$ as long as
\eqref{BKM} holds. This concludes the proof of Theorem \ref{thm:local}.

\section{Critical threshold}\label{sec:threshold}
In this section, we discuss when the criterion \eqref{BKM} holds
globally in time. We start with expressing the dynamics of $d:=\pa_xu$
by differentiate \eqref{eq:main} in $x$:
\[\pa_td+(1-2u)e^{-\ubar}\pa_xd+e^{-\ubar}\big(-2d^2+(3u-5u^2)d+(u^3-u^4)\big)=0.\]
Together with \eqref{eq:uchar}, we get a coupled dynamics of $(d, u)$
along characterstic paths.
\begin{equation}\label{eq:dynamics}
 \begin{cases}
  \dot{d}=\big(2d^2-(3u-5u^2)d-u^3(1-u)\big)e^{-\ubar},\\
  \dot{u}=-u^2(1-u)e^{-\ubar}.
 \end{cases}
\end{equation}
Here, $\dot{f}$ denotes the material derivative of $f$,
\[\dot{f}(t,X(t))=\frac{d}{dt}f(t,X(t))=\pa_tf+((1-2u)e^{-\ubar})\pa_xf.\]

Note that a classical sufficient condition to avoid the breakdown of the
characteristics is that the velocity field is Lipschitz.
\[\left\|\pa_x\big((1-2u)e^{-\ubar}\big)\right\|_{L^\infty}=
  \left\|\big(-2\pa_xu+(1-2u)u\big)e^{-\ubar}\right\|_{L^\infty}
  \leq 1+2\|\pa_xu\|_{L^\infty}.\]
Therefore, as long as condition \eqref{BKM} holds, the characterstic
paths remains valid.

We now perform a phase plane analysis on $(d, u)$ through each
characteristic path. It is worth noting that $e^{-\ubar}$ is
nonlocal. So the values of $(d, u)$ can not be determined solely by
information along the characteristic path. However, the ratio
\[ \frac{\dot{d}}{\dot{u}}=\frac{2d^2-(3u-5u^2)d-u^3(1-u)}{-u^2(1-u)}\]
is local. Therefore, the trajectories of $(d, u)$ only depend on local
information. If we express a trajectory as a function $d=d(u)$, then
it will satisfy the ODE
\begin{equation}\label{eq:traj}
  d'=\frac{2d^2-(3u-5u^2)d-u^3(1-u)}{-u^2(1-u)}.
\end{equation}

Figure \ref{fig:threshold} illustrates the flow map in the phase
plane. In particular, $(0,0)$ is a degenerated hyperbolic point. There
is an inward trajectory which separates the plane into two region.
The left region will flow towards $(0,0)$, and the right region will
flow towards $d\to\infty$. This indicates the two differernt
behaviors: global boundedness versus blowup, respectively.
This is so called the \emph{critical threshold phenomenon}.

For the rest of this section, we will show such phenomenon rigorously.
This then leads to a proof of Theorem \ref{thm:main}.

\begin{figure}[ht]
  \includegraphics[width=.8\linewidth]{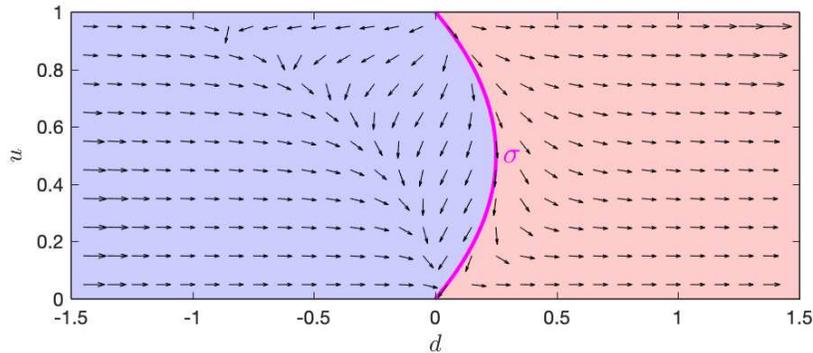}
  \caption{The flow map and the critical threshold in $(d,u)$-plane}\label{fig:threshold}
\end{figure}

\subsection{The sharp critical threshold}
We define the critical threshold that distinguishes the two regions in
Figure \ref{fig:threshold} as $d=\sigma(u)$.
The function $\sigma: [0,1]\to\R$ should satisfy the following ODE
\begin{equation}\label{eq:sigma}
  \sigma'(x)=\frac{2\sigma^2-(3x-5x^2)\sigma
    -x^3(1-x)}{-x^2(1-x)},\quad   \sigma(0)=0.
\end{equation}
In particular, $\sigma'(0)$ can be determined by
\begin{align*}\sigma'(0)=&\lim_{x\to0}\frac{2\sigma(x)^2-(3x-5x^2)\sigma(x)
                           -x^3(1-x)}{-x^2(1-x)}\\
  =&-2\left(\lim_{x\to0}\frac{\sigma(x)}{x}\right)^2+
     3\lim_{x\to0}\frac{\sigma(x)}{x}=-2\sigma'(0)^2+3\sigma'(0).
\end{align*}
This implies $\sigma'(0)=1$.

Therefore, \eqref{eq:sigma} uniquely defines a function
$\sigma$.

\subsection{Global regularity for subcritical initial data}
We now prove the first part of Theorem \ref{thm:main}. The goal is to
show that, if the initial data satisfy \eqref{sub}, then condition
\eqref{BKM} holds for any time $T$. Equivalently, we will show
$d=\pa_xu$ is bounded along all characterstic paths.

First, we show an upper bound of $d$.
\begin{proposition}[Invariant region]
  Let $(d, u)$ satisfy the dynamics \eqref{eq:dynamics} with initial condition
  $d_0\leq\sigma(u_0)$.
  Then, $d(t)\leq\sigma(u(t))$ for any time $t\geq0$.
\end{proposition}
\begin{proof}
  We first consider two special cases $u_0=0$ and $u_0=1$. In both
  cases, $\dot{u}=0$ and hence $u$ does not change in time.

  For $u_0=0$, the dynamics of $d$ becomes
  \begin{equation}\label{eq:u0}
    \dot{d}=2d^2e^{-\ubar}.
  \end{equation}
  If $d_0\leq\sigma(0)=0$, clearly $d(t)\leq0$ for any $t\geq0$.

  For $u_0=1$. the dynamics of $d$ becomes
    \begin{equation}\label{eq:u1}
      \dot{d}=2d(d+1)e^{-\ubar}.
    \end{equation}
  Again, if $d_0\leq\sigma(1)=0$, then $d(t)\leq0$ for any $t\geq0$.
  
  Next, we consider the case $u_0\in(0,1)$. Here, we use the fact that
  trajectories do not cross. To be more precise, we argue by
  a contradiction. Suppose there exists a time $t$ such that
  $d(t)>\sigma(u(t))$. Then, there must exist a time $t_0$ so that the
  $(d,u)$ first exit the region at $t_0+$. By continuity,
  $d(t_0)=\sigma(u(t_0))$.
  Starting from $(d(t_0),u(t_0))$, the trajectory satisfies
  \eqref{eq:traj}.

  By definition \eqref{eq:sigma}, $d=\sigma(u)$ is a
  solution in the phase plane. The standard Cauchy-Lipschitz theorem
  ensures that \eqref{eq:traj} with initial condition $(d(t_0),u(t_0))$
  has a local unique solution. Therefore, the solution has to be
  $d(t_0+)=\sigma(u_0(t_0+))$. This contradicts
  the assumption that $(d,u)$ exit the region at $t_0+$.
\end{proof}

Next, we show a lower bound of $d$. This can be easily observed by
Figure \ref{fig:threshold}, as the flow is moving to the right as long
as $d<-1$.
\begin{proposition}
   Let $(d, u)$ satisfy the dynamics \eqref{eq:dynamics}. Then, for
   any $t\geq0$,
   \[d(t)\geq\min\{-1, d_0\}.\]
\end{proposition}
\begin{proof}
  We rewrite
  \[\dot{d}=2(d-d_-)(d-d_+)e^{-\ubar},\quad
    d_\pm=\frac{(3u-5u^2)\pm\sqrt{(3u-5u^2)^2+8u^3(1-u)}}{4}.\]
  Then, $\dot{d}\geq0$ if $d\leq d_-$. This implies
  $d(t)\geq\min\{d_-, d_0\}$. Note that for $u\in[0,1]$, $d_-\geq1$.
  Therefore, we obtain the lower bound.  
\end{proof}

Combining the two bounds, we know that along each characteristic path,
$d$ is bounded in all time. Collecting all characterstic paths, we
obtain $\|\pa_xu(t,\cdot)\|_{L^\infty}$ is bounded for any $t\geq0$.
Global regularity then follows from Theorem \ref{thm:local}.

\subsection{Finite time breakdown for supercritical initial data}
We turn to prove the second part of Theorem \ref{thm:main}.
Suppose the initial data satisfy \eqref{super}. Then, we consider the
characteristic path originated at $x_0$, namely
$d_0=u_0'(x_0)$ and $u_0=u_0(x_0)$. So,
\begin{equation}\label{eq:superinit}
  d_0>\sigma(u_0).
\end{equation}

For $u_0=0$ or $u_0=1$, finite time blow up can be easily obtain by
the Ricatti-type dynamics \eqref{eq:u0} and \eqref{eq:u1}. Moreover,
as $0\leq u\leq1$, we must have $d_0=0$ when $u_0=0$ or
$1$. Therefore, there is no supercritical data with $u_0=0$ or $1$.

We focus on the case when $u_0\in(0,1)$. The main idea is illustrated
in Figure \ref{fig:blowup}. For each trajectory starting at a
supercritical initial point $(d_0,u_0)$, $u$ is getting close to 0 as
time evolves, unless blowup already happens. When $u$ becomes close to 0,
the dynamics of $d$ becomes close to \eqref{eq:u0}. Then, if $d$ is
away from 0, the Ricatti-type dynamics will lead to a finite time blowup.

\begin{figure}[ht]
  \includegraphics[width=.6\linewidth]{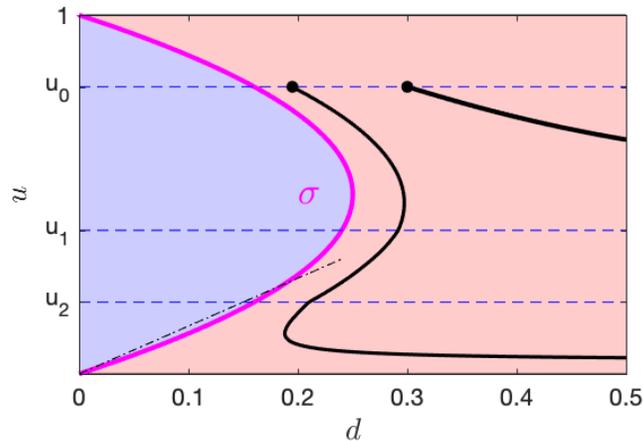}
  \caption{An illustration of typical trajectories with supercritical
    initial data $(d_0, u_0)$. Case 1: blow up happens before the
    trajectory reaches $u_1$. Case 2: the trajectory passes $u_1$, but
  blow up eventually happens in finite time.}\label{fig:blowup}
\end{figure}

To rigorously justify the idea, we first examine the dynamics of $u$
in \eqref{eq:dynamics}.

\begin{proposition}\label{prop:uto0}
  Let $(d,u)$ be a solution of \eqref{eq:dynamics} with supercritical
  initial data $(d_0, u_0)$. Then, for any $u_1\in(0,u_0)$, there
  exists a finite time $t_1$ such that, either $d(t)\to\infty$ before
  $t_1$, or $u(t_1)\leq u_1$.
\end{proposition}
\begin{proof}
  Using the bound on the nonlocal term \eqref{eq:nonlocalLinf}, we get
  \[\dot{u}\leq-e^{-m}u^2(1-u).\]
  As long as $(d, u)$ is bounded, the characterstic path stays valid. 

  The following comparison principle holds. Let $\eta=\eta(t)$ satisfy the ODE
  \begin{equation}\label{eq:eta}
    \eta'=-e^{-m}\eta^2(1-\eta),\quad \eta(0)=u_0.
  \end{equation}
  Then, $u(t)\leq\eta(t)$. Indeed,
  \[\dot{u}(t)-\eta'(t)\leq
    e^{-m}\left(-u^2(1-u)+\eta^2(1-\eta)\right)
    \leq3e^{-m}(u-\eta).\]
  This implies
  \[u(t)-\eta(t)\leq (u(0)-\eta(0))e^{3e^{-m}t} \leq0.\]

  The dynamics $\eta$ in \eqref{eq:eta} can be solved explicitly
  \[\left.\left(\frac{1}{\eta}+\log\frac{1-\eta}{\eta}\right)\right]_{u_0}^{\eta(t)}=e^{-m}t.\]
 Therefore, $\eta(t_1)=u_1$ at
 \[t_1=e^m\left(\frac{1}{u_1}+\log\frac{1-u_1}{u_1}-\frac{1}{u_0}-\log\frac{1-u_0}{u_0}\right)<+\infty.\]
  Applying the comparison principle, we end up with $u(t_1)\leq u_1$.
\end{proof}
Proposition \ref{prop:uto0} distinguishes the two cases illustrated in
Figure \ref{fig:blowup}. Either blowup happens before $u$ reaches
$u_1$, which takes finite time, or the trajectory passes $u_1$. We
shall focus on the latter case from now on.

Next, we argue that by picking a small enough $u_1>0$, the dynamics
\eqref{eq:dynamics} will lead to a blowup in finite time, as long as
$d$ stays away from zero.
\begin{proposition}\label{prop:dtoinf}
  Let $(d,u)$ be a solution of \eqref{eq:dynamics}.
  Suppose $d$ is uniformly bounded away from zero, namely there exists a $C_*>0$
  such that
  \begin{equation}\label{eq:dlow}
    d(t)\geq C_*,\quad \forall~t\geq0.
  \end{equation}
  Then, there exists a $u_1>0$, depending on $C_*$, such that, with the initial condition
  $(d(t_1), u(t_1)=u_1)$, the solution has to blow up in finite time.
\end{proposition}
\begin{proof}
  As $u(t_1)=u_1$, we know $u(t)\leq u_1$ for
  any $t\geq t_1$. Then, we get
  \[\dot{d}\geq
    e^{-\ubar}(2d^2-3u_1 d-u_1^3)=2e^{-\ubar}(d-d_-)(d-d_+),\quad
    d_\pm=\frac{3\pm\sqrt{9+8u_1}}{4}u_1.\]

  Pick $u_1=C_*/4$, then   \[d(t_1)\geq C_*=4u_1>2d_+.\]
  This implies $d(t)>2d_+$ for all $t\geq t_1$. We can then use
  \eqref{eq:nonlocalLinf} to bound the nonlocal term and get
  \begin{equation}\label{eq:dblow}
    \dot{d}\geq 2e^{-m}(d-d_-)(d-d_+),\quad\forall~t\geq t_1.
  \end{equation}
  Then, by a comparison principle (similar as the one used in
  Proposition \ref{prop:uto0}), the solution
  \[d(t)\geq
    \frac{d_-e^{2e^{-m}(d_+-d_-)(t-t_1)}(d(t_2)-d_+)-d_+(d(t_1)-d_-)}
    {e^{2e^{-m}(d_+-d_-)(t-t_1)}(d(t_1)-d_+)-(d(t_1)-d_-)},\]
  where the right hand side is the exact solution of the ODE
  \eqref{eq:dblow}  with an equal sign. It blows up at
  \[T_*=t_1+\frac{1}{2e^{-m}(d_+-d_-)}\log\frac{d(t_1)-d_-}{d(t_1)-d_+}
  <t_1+\frac{2e^m}{C_*}<+\infty.\]
   Therefore, $d$ has to blow up no later than $T_*$.
 \end{proof}

 We are left to show the uniform lower bound on $d$, \textit{i.e.} condition
 \eqref{eq:dlow}, for any supercritical initial data. We shall work
 with trajectories in the phase plane.

 Let us denote $d=d(u)$ be the trajectory that go through $(d_0,u_0)$.
 As both $d$ and $\sigma$ satisfy \eqref{eq:traj}, we compute
  \[(d(u)-\sigma(u))'=\frac{2(d(u)+\sigma(u))-(3u-5u^2)}{-u^2(1-u)}(d(u)-\sigma(u))=:
    A(u)(d(u)-\sigma(u)).\]
  Since $(d_0, u_0)$ satisfy \eqref{eq:superinit}, we get $d(u_0)-\sigma(u_0)>0$.  
 $A(u)$ is bounded as long as $u$ stays away from 0 and 1. Therefore,
 we obtain
 \[d(u)\geq\sigma(u)\geq0,\quad\forall~u\in(0,1).\]

 Moreover, for any $u\in(0,u_0)$, we can estimate $A$ by
  \[A(u)\leq \frac{3-5u}{u(1-u)} \leq \frac{3}{u}.\]
  Integrating in $[u, u_0]$, we get
  \begin{equation}\label{eq:dbound}
    d(u)\geq d(u)-\sigma(u)=(d_0-\sigma(u_0))\exp\left[-\int_u^{u_0}A(u)du\right]
    \geq
    \frac{(d_0-\sigma(u_0)) }{u_0^3}u^3.
  \end{equation}

  Unfortunately, this bound is not uniform in $(0, u_0]$. We need an enhanced
  estimate.

  Let $u_2>0$ such that
  \begin{equation}\label{eq:sigmaboost}
    \sigma(u)\geq\frac{3}{4}u,\quad \forall~u\in[0, u_2].
  \end{equation}
  Note that such $u_2$ exists as $\sigma'(0)=1$.

  For $u\in(0, u_2]$, using \eqref{eq:sigmaboost}, we obtain an improved
  estimate on $A$ as follows.
  \[A(u)\leq \frac{4\sigma(u)
      -(3u-5u^2)}{-u^2(1-u)} \leq \frac{3u
      -(3u-5u^2)}{-u^2(1-u)} = \frac{-5}{1-u}\leq -5.\]
  Since $A(u)$ is negative, we immediately get
  \[
    d(u)\geq d(u)-\sigma(u)\geq d(u_2)-\sigma(u_2),\quad\forall~u<u_2,~u\in
    \textnormal{Dom}(d).
  \]
  This, together with \eqref{eq:dbound}, shows a uniformly lower bound on $d$
  \[d(u)\geq \frac{d_0-\sigma(u_0)}{u_0^3}u_2^3,\quad \forall~u\leq
    u_0,~u\in \textnormal{Dom}(d).\]

  Condition \eqref{eq:dlow} follows immediately, with
  $C_*=(d_0-\sigma(u_0))u_2^3u_0^{-3}$.

\section{Examples and simulations}\label{sec:exp}
In this section, we present examples and numerical simulations to
illustrate our main critical threshold result, Theorem \ref{thm:main}. 

The numerical method we use is the standard finite volume scheme, with
a large enough computational domain. One can consult
\cite{kurganov2009non} for an extensive discussion on the
numerical implementation.

We shall also compare the numerical results for the three different
types of nonlocal interaction kernels. Recall
\begin{equation}\label{eq:kernels}
  K(x)=
  \begin{cases}
    0, & \text{{\color{blue}\ding{172}}  LWR model: look-ahead distance }L=0, \\
    1_{[-1, 0)} (x), & \text{{\color{magenta}\ding{173}}  SK model: look-ahead distance }L=1,\\
    1_{(-\infty,0]} (x), &  \text{{\color{red}\ding{174}}  Our model: look-ahead distance }L=\infty,\\
    1, & \text{\ding{175}  LWR model: globally uniform kernel}, \\
   \end{cases}
\end{equation}
Here, $1_{A}$ denotes the indicator function of a set $A$.

\subsection{Supercritical initial data}
Many smooth initial data $u_0$ lie in the supercritical region
\eqref{super}. In particular, we argue that all compactly supported smooth
function lies in the supercritical region.

\begin{proposition}
Let $u_0\in C^1(\R)$ is non-negative and compactly supported. Then,
$u_0$ satisfies the
supercritical condition \eqref{super}.
\end{proposition}\label{prop:compact}
\begin{proof}
  We argue by contradiction. Suppose $u_0$ lies in the subcritical
  region. Then, we have
  \begin{equation}\label{eq:compact}
    u_0'(x)\leq u_0(x),\quad\forall~x\in\R.
  \end{equation}  
  Let $x_L$ be the left end point of the support of $u_0$, namely
  \[x_L=\arg\inf_x\{u_0(x)>0\}.\]
  By continuity, we know $u_0(x_L)=0$. Solving the ODE
  \eqref{eq:compact} with initial condition at $x_L$ yields
  \[u_0(x)\leq0,\quad\forall~x\geq x_L.\]
  This contradicts with the definition of $x_L$.
  Hence, $u_0$ can not lie in the subcritical region. It must be
  supercritical.
\end{proof}

As an example, let us take the following smooth and compactly
supported initial data.
\begin{equation}\label{eq:compactinit}
u_0 (x)=
  \begin{cases}
    e^{-\frac{1}{1-x^2}},   & |x|<1, \\
    0,   & |x| \geq 1.
  \end{cases}
\end{equation}
Figure \ref{fig:sup_threshold} shows the contour plot of $(u_0'(x),
u_0(x))$ in the phase plane for all $x\in\R$. Clearly, the curve does not lie in the
subcritical region. So, $u_0$ is supercritical. Theorem \ref{thm:main}
then implies a finite time wave break-down.

\begin{figure}[ht]
  \includegraphics[width=.8\linewidth]{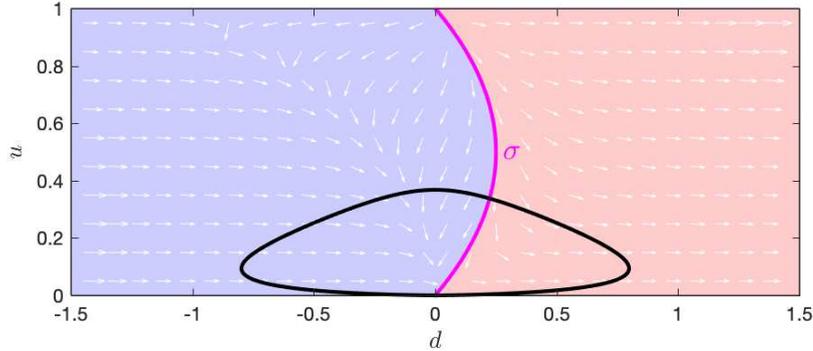}
  \caption{The contour plot of $(u'_0(x), u_0(x))$ in the phase
    plane where $u_0$ is \eqref{eq:compactinit}.
    This initial condition lies in the supercritical region.}\label{fig:sup_threshold}
\end{figure}

Figure \ref{fig:sup_compare} shows the numerical result for the
model with initial data \eqref{eq:compactinit}, together with other
models. The wave break-down can be easily observed, which matches our
theoretical result.

\begin{figure}[ht]
  \includegraphics[width=.8\linewidth]{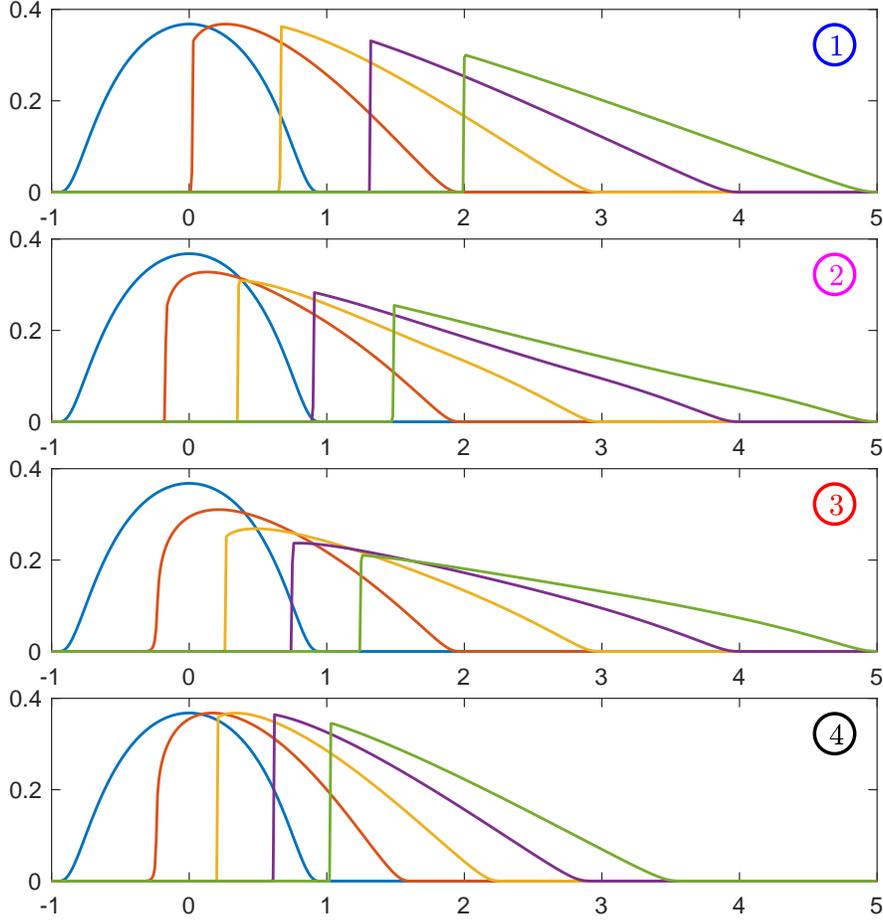}
  \caption{Snapshots of solutions for the dynamics for the four
    kernels, with supercritical
    initial condition \eqref{eq:compactinit} at time $t=0,
    1, 2, 3, 4$.}\label{fig:sup_compare}
\end{figure}

Note that since
\begin{equation}\label{eq:kernelcomp}
  0\leq 1_{[-1.0)}(x) \leq 1_{(-\infty,0]}(x)\leq
1,\quad\forall~x\in\R,
\end{equation}
model \ding{172} has the fastest wave speed, while model \ding{175} has the
slowest. This is indeed captured in the numerical result.

\subsection{Subcritical initial data}\label{sec:expsub}
We now construct an initial condition $u_0$ that lies in the subcritical
region \eqref{sub}. 

Due to Proposition \ref{prop:compact}, $u_0$ can not be compactly
supported. Moreover, we need $u_0\in L^1(\R)$. One valid choice is that
$u_0$ decays algebraically when $x\to-\infty$, namely
$u_0(x)\sim(-x)^{-\beta}$ for $\beta>1$. We can check
\[\lim_{x\to-\infty}\frac{u_0'(x)}{u_0(x)}=
\lim_{x\to-\infty}\frac{\beta(-x)^{-\beta-1}}{(-x)^{-\beta}}=0<1.\]
Therefore, $(u_0'(x), u_0(x))$ should lie in the subcritical region of the
phase plane when $x$ is very negative.

For large $x$, the choice of $u_0$ is less critical. As long as
$u_0'(x)\leq0$, it always lies in the subcritical region. We can
either choose $u_0$ vanishes for large $x$, or it decays fast as $x\to+\infty$.

Here is a subcritical initial condition
\begin{equation}\label{eq:subinit}
u_0 (x)=
  \begin{cases}
    1/x^2,   & x \in (-\infty, -3], \\
    (3x^5 + 35x^4 + 123x^3 + 81x^2 -162x +162)/1458,   & x \in (-3, 0], \\
    e^{-x}/9, &   x \in (0, \infty).
  \end{cases}
\end{equation}
The middle part is chosen as a polynomial which smoothly connects the
two functions, so that $u\in C^2(\R)$.

The contour plot of $(u_0'(x),u_0(x))$ is shown in Figure
\ref{fig:sub_threshold}, which indicates $u_0$ is subcritical.
Therefore, as a consequence of Theorem \ref{thm:main}, the solution
should be globally regular.

\begin{figure}[ht]
  \includegraphics[width=.8\linewidth]{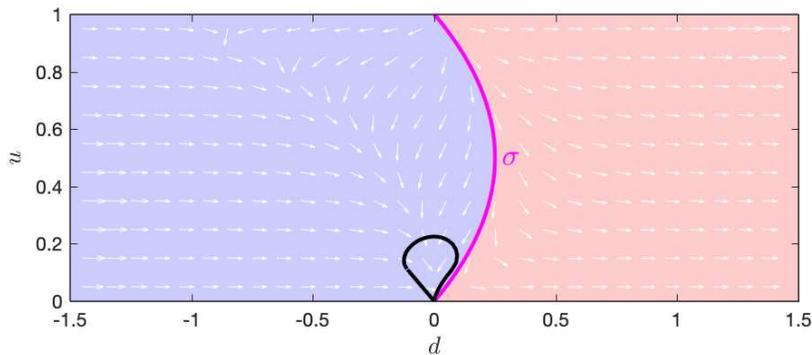}
  \caption{The contour plot of $(u'_0(x), u_0(x))$ in the phase
    plane where $u_0$ is \eqref{eq:subinit}.
    This initial condition lies in the subcritical region.}\label{fig:sub_threshold}
\end{figure}

Figure \ref{fig:sub_compare} shows the numerical results for all four
models with initial conditon \eqref{eq:subinit}. We observe that the
solution of our model \ding{174} indeed does not generate shocks. 

The wave speeds of the four models behave similar as the supercritical
case, due to \eqref{eq:kernelcomp}.
However, very interestingly, our model \ding{174}  is the \emph{only}
model where there is no  finite time wave break-down.
Indeed, we plot the quantity
$\|\pa_xu(\cdot,t)\|_{L^\infty}/\|u(\cdot,t)\|_{L^\infty}$ against
time $t$ in Figure \ref{fig:sub_blowup}. The quantity blows up in
finite time for models \ding{172}, \ding{173} and \ding{175}, but
remains bounded for our model \ding{174}. 
\begin{figure}[ht]
  \includegraphics[width=.8\linewidth]{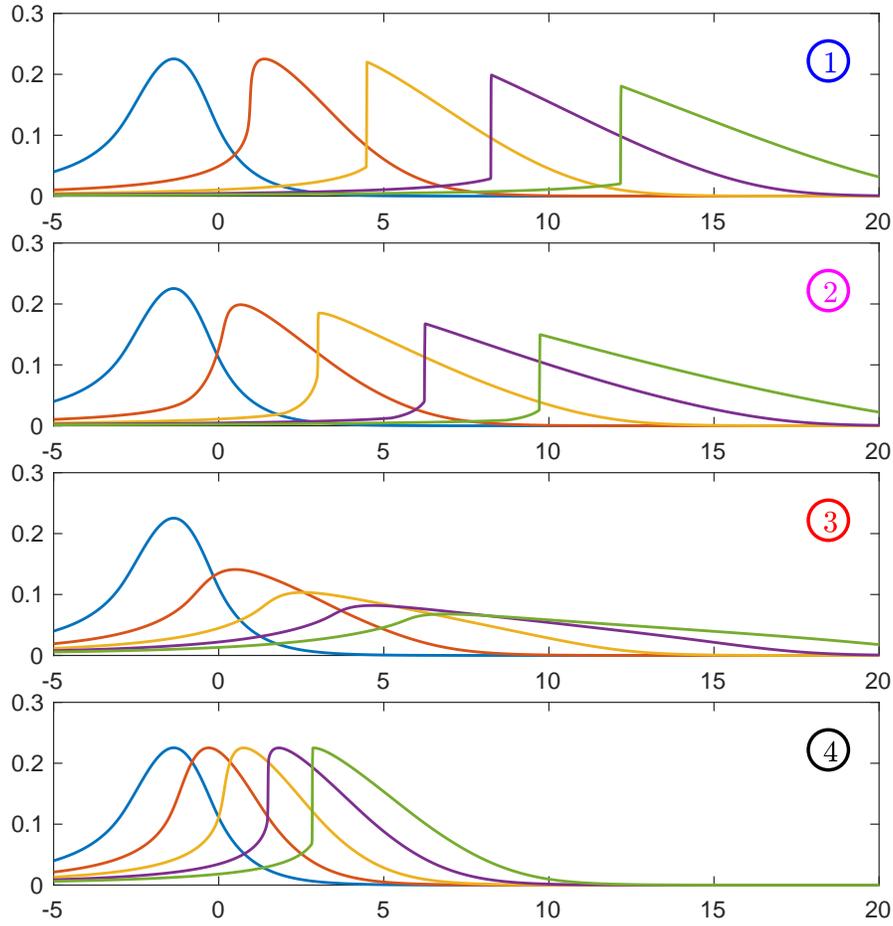}
  \caption{Snapshots of solutions for the dynamics for the four
    kernels, with subcritical
    initial condition \eqref{eq:subinit} at time $t=0,
    5, 10, 15, 20$.}\label{fig:sub_compare}
\end{figure}

\begin{figure}[h]
  \includegraphics[width=.9\linewidth]{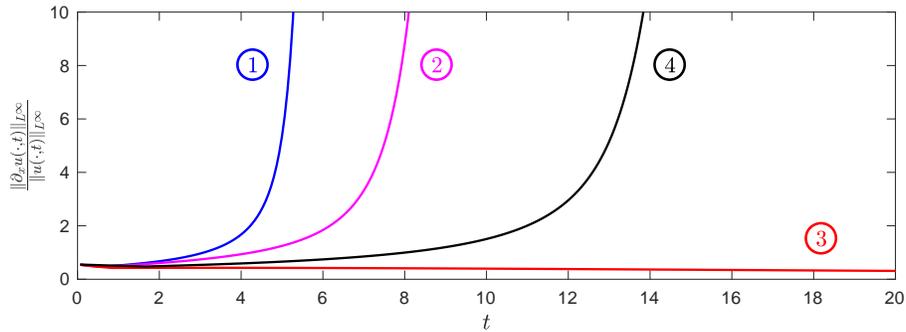}
  \caption{Numerical indicators of finite time blowup versus global
    regularity. With initial condition \eqref{eq:subinit}, only our
    kernel \ding{174} has a global smooth solution.}\label{fig:sub_blowup}
\end{figure}

\section{Further discussion}\label{sec:close}
We have shown a sharp critical threshold for our traffic
model \eqref{traffic_main} with look-ahead kernel \eqref{ourubar}. We
also compare our model with other classical kernels
\eqref{eq:kernels} through numerical simulations.
Our kernel has a unique feature that the solution
remains globally regular for initial conditions like \eqref{eq:subinit}.

To understand such behavior, we shall focus on the nonlocal slow down
factor $e^{-\ubar}$. From \eqref{eq:kernelcomp}, we observe that our
model has a factor which is neither the largest nor the
smallest. Hence, the size of the slow down factor does not matter.

An important feature of our model is that, the slow down factor is
monotone increasing. Indeed, we have
\[\pa_xe^{-\ubar}=u e^{-\ubar}>0,\quad\forall~x\text{~~s.t.~~}u(x)>0.\]
This implies that the front crowd does not slow down as much as the
back crowd. This could help avoid the shock formation, as observed in
the example.

For general nonlocal look-ahead kernel, it remains open whether there
are subcritical initial data which lead to global regularity.
If we consider a family of kernel $K_L$ in \eqref{scaling}, our result
indicates that subcritical initial data exist for $L=\infty$. On the
other hand, subcritical initial data does not exist for the LWR model,
where $L=0$. For $L\in(0,\infty)$, the problem is open.
A conjecture is, subcritical initial data exists for $L$ large
enough. This will be left for future investigation.

\newpage~\newpage

\begin{appendix}
  \section{Composition estimate}\label{sec:app}

In this section, we show the following estimate on the composition of
two functions. The estimate is useful to control the nonlocal weight
$e^{-\ubar}$ for our system.
  
\begin{theorem}\label{thm:composition}
  Let $s>0$. Suppose $g\in L^\infty\cap \dot{H}^s(\R)$ and
  $f\in C^{\lceil s \rceil}(\textnormal{Range}(g))$.
 Then, the composition
  $f\circ g \in L^\infty\cap \dot{H}^s(\R)$. Moreover, there exists a
  constant $C>0$, depending on $s, \|f\|_{C^{\lceil s \rceil}}$ and
  $\|g\|_{L^\infty}$, such that 
  \[\|f\circ g\|_{\dot{H}^s}\leq C\|g\|_{\dot{H}^s}.\]
\end{theorem}
\begin{proof}
We first consider the case when $s$ is an integrer. Express
$\pa_x^s(f(g(x)))$ using Fa\`a di Bruno's formula

  \[\pa_x^s(f(g(x)))=\sum_{\alpha\in S_s} C_\alpha(x)\prod_{r=1}^s
   (\pa_x^rg(x))^{\alpha_r} .\]
where
  \[S_s=\left\{\alpha=(\alpha_1,\cdots, \alpha_s)~ :~
      a_k\in\mathbb{N},~~
      \sum_{r=1}^sr\alpha_r=s,~~\sum_{r=1}^s\alpha_r\leq s.\right\}.\]
  and
  \[C_\alpha(x)=s!\prod_{r=1}^s
    \left(\frac{1}{\alpha_r!\cdot
      (r!)^{\alpha_r}}\right) \pa_x^{\nu(\alpha)}f (g(x)),\quad \nu(\alpha)=\sum_{r=1}^s\alpha_r.\]
Then,
\[\|\pa_x^s(f\circ g)\|_{L^2}\leq\sum_{\alpha\in S_s}
  \|C_\alpha\|_{L^\infty}\left\|\prod_{r=1}^s(\pa_x^rg)^{\alpha_r}\right\|_{L^2}
  \lesssim \|f\|_{C^s}\sum_{\alpha\in S_s}\left\|\prod_{r=1}^s(\pa_x^rg)^{\alpha_r}\right\|_{L^2}. \]

Now, we estimate the last term. Applying H\"older's inequality, we get
\[\left\|\prod_{r=1}^s(\pa_x^rg)^{\alpha_r}\right\|_{L^2}\leq
\prod_{r=1}^s\|(\pa_x^rg)^{\alpha_r}\|_{L^{p_r}}=
\prod_{r=1}^s\|\pa_x^rg\|_{L^{\alpha_rp_r}}^{\alpha_r},\]
where $\{p_r\}_{r=1}^s$ is chosen as $p_r=\frac{2s}{r\alpha_r}$. So we have
\[\sum_{r=1}^s\frac{1}{p_r}=\frac{1}{2s}\sum_{r=1}^sr\alpha_r=\frac{1}{2}.\]
For each term $\|\pa_x^rg\|_{L^{\alpha_rp_r}}$, we apply
Gagliardo-Nirenberg-Sobolev interpolation inequality
\[\|\pa_x^rg\|_{L^{\alpha_rp_r}}=\|\pa_x^rg\|_{L^{\frac{2s}{r}}}\lesssim
  \|\pa_x^sg\|_{L^2}^{\frac{r}{s}}\|g\|_{L^\infty}^{1-\frac{r}{s}}.\]
Collecting all terms together, we obtain
\[\prod_{r=1}^s\|\pa_x^rg\|_{L^{\alpha_rp_r}}^{\alpha_r}
  \lesssim\|\pa_x^sg\|_{L^2}^{\sum_{r=1}^s\alpha_r\frac{r}{s}}
  \|g\|_{L^\infty}^{\sum_{r=1}^s\alpha_r(1-\frac{r}{s})}=
  \|\pa_x^sg\|_{L^2}\|g\|_{L^\infty}^{\nu(\alpha)-1}.\]
This concludes the proof.

Next, we discuss the case when $s$ is not an integer.
For $s\in(0,1)$, one can directly apply the chain rule for fractional
derivatives \cite[Proposition 3.1]{christ1991dispersion}
\[\|f\circ g\|_{\dot{H}^s}\leq C\|\pa_xf\|_{L^\infty}\|g\|_{\dot{H}^s}.\]
where $C$ is a constant depending on $s$ and $\|g\|_{L^\infty}$.

For $s>1$, we can combine the estimate for $\lfloor s\rfloor$ and the
fractional chain rule for $s-\lfloor s\rfloor$. The detail will be
left to the readers.
\end{proof}

\end{appendix}

\bibliographystyle{plain}
\bibliography{traffic}

\end{document}